\documentclass{ijuc}
\usepackage{lineno}
\usepackage{cite}
\usepackage{amsmath,amssymb,amsfonts,url,bm,enumerate}
\usepackage{amsthm}
\usepackage{graphicx}
\usepackage{textcomp}
\usepackage[dvipsnames]{xcolor}
\usepackage{colortbl}
\usepackage{tikz}
\usetikzlibrary{arrows.meta}
\usepackage{float}
\usepackage{booktabs}
\usepackage{array}
\usepackage{multicol}
\def\BibTeX{{\rm B\kern-.05em{\sc i\kern-.025em b}\kern-.08em
    T\kern-.1667em\lower.7ex\hbox{E}\kern-.125emX}}
    
\usepackage{hyperref}
\usepackage{pict2e}

\definecolor{myblue}{RGB}{135, 206, 235}

\definecolor{darkgreen}{rgb}{0.0, 0.5, 0.0}

\definecolor{aliceblue}{rgb}{0.94, 0.97, 1.0}
\definecolor{deepmagenta}{rgb}{0.8, 0.0, 0.8}
\definecolor{inchworm}{rgb}{0.7, 0.93, 0.36}
\definecolor{airforceblue}{rgb}{0.36, 0.54, 0.66}
\definecolor{violet}{rgb}{0.53, 0.0, 0.69}
\definecolor{rosewood}{rgb}{0.4, 0.0, 0.04}
\definecolor{olive}{rgb}{0.5, 0.5, 0.0}
\definecolor{navajowhite}{rgb}{1.0, 0.87, 0.68}
\definecolor{otterbrown}{rgb}{0.4, 0.26, 0.13}
\definecolor{orange-red}{rgb}{1.0, 0.27, 0.0}
\definecolor{mintcream}{rgb}{0.96, 1.0, 0.98}
\definecolor{cyan}{rgb}{0.0, 1.0, 1.0}
\definecolor{aqua}{rgb}{0.0, 1.0, 1.0}
\definecolor{aquamarine}{rgb}{0.5, 1.0, 0.83}
\definecolor{bubbles}{rgb}{0.91, 1.0, 1.0}
\definecolor{azure}{rgb}{0.0, 0.5, 1.0}
\definecolor{dartmouthgreen}{rgb}{0.05, 0.5, 0.06}
\definecolor{deepjunglegreen}{rgb}{0.0, 0.29, 0.29}
\definecolor{brightgreen}{rgb}{0.4, 1.0, 0.0}

\theoremstyle{plain}
\newtheorem{theorem}{Theorem}
\newtheorem{lemma}[theorem]{Lemma}
\newtheorem{proposition}[theorem]{Proposition}

\newtheorem{corollary}[theorem]{Corollary}

\theoremstyle{definition}
\newtheorem{definition}[theorem]{Definition}

\definecolor{darkgreen}{rgb}{0.0, 0.5, 0.0}

\definecolor{brightgreen}{rgb}{0.4, 1.0, 0.0}
\newcolumntype{?}{!{\vrule width 1pt}}

\DeclareMathOperator{\Max}{Max}
\DeclareMathOperator{\Min}{Min}

\begin{document}

\title{Quantum implications in orthomodular posets\\
}
\author{Kadir~Emir\inst{1}\email{emir@math.muni.cz} \and %
Jan~Paseka\inst{1}\email{paseka@math.muni}}
\institute{{Department of Mathematics and Statistics} \\
	\textit{Faculty of Science, Masaryk University}\\
	Brno, Czech Republic}

\maketitle

\begin{abstract}
We show that, for every orthogonal lub-complete poset $\mathbf P$, we can introduce 
multiple-valued  implications  sharing
properties with quantum implications presented for orthomodular lattices by 
Kalmbach. We call them {\em classical implication}, {\em Kalmbach implication}, {\em non-tolens implication}, 
{\em Dishkant implication} and 
{\em Sasaki implication}. 

If the classical implication satisfies the {\em order property}, then the corresponding orthologic 
becomes classical  and vice versa. If the Kalmbach  or  non-tolens or Dishkant  or Sasaki implication 
meets the {order property}, then the corresponding orthologic 
becomes quantum  and vice versa. A related result for the modus ponens rule is obtained. 

\end{abstract}


\section{Introduction}

Orthomodular posets were introduced,  to  formalize 
event structures of quantum mechanical systems faithfully; see,
e.g.,\ \cite{B} and \cite{OML} for details. Note  that a bounded 
poset $(P,\leq,{}',0,1)$ with complementation which is an antitone involution, is called
{\em orthomodular} if it is orthogonal, i.e.,\ $x\vee y$ exists for $x,y\in P$
whenever $x\leq y'$, and it satisfies the orthomodular law
\begin{align*}
	x\leq y\ \text{implies}\ x\vee(y\wedge x')=y.\tag{OM}
\end{align*}

The implication is the most productive connective in every logic 
because it enables one to derive new propositions from 
given ones.
But it is not a straightforward task to implement the connective implication in the propositional logic based on an orthomodular poset. The main reason is that orthomodular posets are only partial lattices where the lattice operation join need not exist for elements being neither comparable nor orthogonal. Hence implication is considered a partial operation. In this case, the chosen avenue may cause problems in developing the corresponding logic.

In \cite{CL22}, Chajda and L\"anger developed another approach, where the considered orthomodular posets are of finite height. They introduced an everywhere-defined connective implication that satisfies the properties usually asked in classical and non-classical logics. Moreover, they could 
 to derive a Gentzen system for the axiomatization of this logic.

Namely, the implication $x\rightarrow y$ for given entries $x$ and $y$ need not be an element. Still, it   
may be a {\em subset of the orthomodular poset} in question such that this subset contains only 
the {\em maximal values}. Hence its  elements are mutually incomparable. It means that one cannot
prefer one element of this set with respect to the other elements. 
If the orthomodular poset is a lattice then the implication defined there coincides with the implication usually introduced in orthomodular lattices, e.g., \cite{OML} and \cite{MP}. In the case when several maximal values of $x\rightarrow y$ exist, then this may not cause problems since such  ``unsharp'' reasoning in quantum mechanics can be found, e.g.,\ in \cite{GG}.

Hence, if orthomodular posets should serve
as an algebraic axiomatization of certain orthologics of quantum mechanics (quantum logics), we
search for a suitable definition of  implication-sharing
properties with quantum implications introduced for orthomodular lattices by 
Kalmbach \cite{OMLo}, 
that are listed in the following table:
\begin{align*} 
	\begin{tabular}{ l l }
		Symbol \,\,\, Definition & Beran's number, name \\ \midrule
		\cellcolor{LimeGreen!50} $a \rightarrow_C b = a' \vee b$ & \cellcolor{LimeGreen!50} 94, \,\, Classical arrow \\
		\cellcolor{Cerulean!50} $a \rightarrow_S b =a' \vee (a \wedge b)$ & \cellcolor{Cerulean!50} 78, \,\, Sasaki arrow \\
		\cellcolor{Cerulean!50} $a \rightarrow_D b = b \vee (a' \wedge b')$ & \cellcolor{Cerulean!50}  46, \,\, Dishkant arrow \\
		\cellcolor{Cerulean!50} $a \rightarrow_K b = (a' \wedge b) \vee (a' \wedge b') \vee (a \wedge (a' \vee b))$ & \cellcolor{Cerulean!50}  30, \,\, Kalmbach arrow \\
		\cellcolor{Cerulean!50} $a \rightarrow_N b = (a \wedge b) \vee (a' \wedge b) \vee (b' \wedge (a' \vee b))$ & \cellcolor{Cerulean!50}  62, \,\, Non-tolens arrow \\
		\cellcolor{Cerulean!50} $a \rightarrow_R b = (a \wedge b) \vee (a' \wedge b) \vee (a' \wedge b')$ & \cellcolor{Cerulean!50}  14, \,\, Relevance arrow \\
	\end{tabular}
\end{align*}

All of these six implications coincide in Boolean algebras. Furthermore, they also coincide in an orthomodular lattice when $a$ and $b$ commutes.

The implications $\rightarrow_i$ where $i \in \{C,S,D,K,N,R\}$ can be characterized as those orthomodular lattice operations which satisfy the Birkhoff - von Neumann criteria
\begin{align}\label{BvN}
	a \rightarrow_i b = 1 \Leftrightarrow a \leq b
\end{align}

Moreover, these six implications give rise to the corresponding quantum
conjunctions $\wedge_i$ and disjunctions $\vee_i$ defined as follows:
\begin{align*}
	a \vee_i b &= a' \rightarrow_i b \\
	a \wedge_i b &= (a \rightarrow_i b')'
\end{align*}
for all  $i \in \{C,S,D,K,N,R\}$.


The paper is structured as follows. 
After this introduction, Section~\ref{Preliminaries} 
provides some notions and notations that will be used in the article. 

In Section~\ref{classical}, we show that an orthologic in which 
$$
\vdash  \phi \rightarrow_C \psi \text{ if and only if } \phi\vdash\psi
$$
holds is a classical logic. In the language of 
 posets with an antitone involution satisfying particular natural conditions,  this means that whenever 
\begin{align*}
a \rightarrow_C b=1 \text{ if and only if } a\leq b\tag{OPC}
\end{align*}
holds, then the poset is a Boolean algebra  and vice versa. We 
then say that the implication satisfies 
the {\em order property}, and the motivation is that the implication must comply 
with the natural underlying order structure of truth values. 

We present a corresponding result for Kalmbach implication $\rightarrow_K$ 
and non-tolens implication $\rightarrow_N$ in Section~\ref{Kalman} and 
for Dishkant implication $\rightarrow_D$ 
and Sasaki implication $\rightarrow_S$ in Section~\ref{SasDis}. 
We show that if 
\begin{align*}
a \rightarrow_K b=1 \text{ if and only if } a\leq b\tag{OPK}
\end{align*}
or 
\begin{align*}
a \rightarrow_N b=1 \text{ if and only if } a\leq b\tag{OPN}
\end{align*}
or 
\begin{align*}
a \rightarrow_D b=1 \text{ if and only if } a\leq b\tag{OPD}
\end{align*}
or 
\begin{align*}
a \rightarrow_S b=1 \text{ if and only if } a\leq b\tag{OPS}
\end{align*}
holds, then the poset is an orthomodular poset and vice versa. 
We close the section with the statement formulated for  posets 
with an antitone involution satisfying particular natural conditions
that an orthologic in which the 
following {\em modus ponens rule} 
\begin{align*}
	\alpha\vdash\beta\quad  \&\quad   %
	\alpha \rightarrow \beta\vdash \phi \rightarrow \psi \Rightarrow \phi\vdash\psi
\tag{MP}
\end{align*}
holds is a quantum logic for $\rightarrow\in \{\rightarrow_K, \rightarrow_N, \rightarrow_D, \rightarrow_S\}$, 
and a classical logic for $\rightarrow\in \{\rightarrow_C\}$ and vice versa.

The above results generalize the work 
of Megill and Pavi\v ci\v c both in classical and quantum situations 
(see \cite{mphpa98,MP}). 

We then conclude in Section~\ref{Conclusions}.

A preliminary short version of the results of this paper was published 
in the Proceedings of the 53rd ISMVL (see \cite{EP}).

In what follows, we take the basic concepts and
results on lattices and posets  for granted. For more
information on these topics, we direct the reader to 
the monographs \cite{Bi} by G.~Birkhoff and \cite{OML} by G.~Kalmbach.

\section{Basic concepts}\label{Preliminaries}

Let $(P,\leq)$ be a poset $a,b\in P$ and $A,B\subseteq P$. We say $A\leq B$ if $x\leq y$ for all $x\in A$ and $y\in B$. Instead of $\{a\}\leq \{b\}$, $\{a\}\leq B$ and $A\leq \{b\}$, 
we simply write $a\leq b$, $a\leq B$ and $A\leq b$, respectively.

The sets \begin{align*} L(A) & :=\{x\in
	P\mid x\leq A\}, \\ U(A) & :=\{x\in P\mid A\leq x\} \end{align*} are called
the {\em lower cone} and {\em upper cone} of $A$, respectively. Instead of
$L(A\cup B)$, $L(A\cup\{b\})$, $L(\{a,b\})$ and $L\big(U(A)\big)$ we write
$L(A,B)$, $L(A,b)$, $L(a,b)$ and $LU(A)$, respectively. Analogously, we
proceed in similar cases. On $2^P$ we introduce two binary relations
$\leq_1$ and $\leq_2$ as follows: 
\begin{align*} A\leq_1B \ &\text{ if and only
		if} \begin{array}{c}\text{ for every }x\in A\text{ there exists some }\\
		y\in B\text{ with }x\leq y,
		\end{array} \\
	A\leq_2B \ &\text{ if and only if} \begin{array}{c}\text{  for every }y\in B\text{ there exists some
	}\\x\in A\text{ with }x\leq y. 	\end{array}  \\ 
\end{align*} A unary operation ${}'$ on $P$ is
called an {\em antitone involution} if \[ x,y\in P\text{ and }x\leq y\text{
	together imply }y'\leq x' \] and if it satisfies the identity \[ x''\approx x.
\] 
Two elements $x,y$ of $P$ are called {\em orthogonal} to each other,
shortly $x\perp y$, if $x\leq y'$ or, equivalently, $y\leq x'$.
By an {\em orthogonal poset}, we mean a bounded poset $(P,\leq,{}',0,1)$ with an antitone involution satisfying the following condition:
\[
x,y\in P\text{ and }x\perp y\text{ imply that }x\vee y\text{ exists}.
\]
Here and in the following, $x\vee y$ denotes the supremum of $x$ and $y$.

The concept of a paraorthomodular poset was introduced in \cite{CFLLP} as a generalization of the concept of a paraorthomodular lattice introduced in \cite{GLP17}. Recall that a bounded {\em poset} $\mathbf P=(P,\leq,{}',0,1)$ with an antitone involution is called {\em paraorthomodular} if it satisfies the following condition:
\begin{align*}
	x,y\in P, x\leq y \text{ and } x'\wedge y=0 \text{ together imply } x=y.\tag{P}
	\end{align*}

 Condition (P) is evidently equivalent with the following condition: 
\begin{align*}
	x,y\in P, x\leq y \text{ and } x\vee y'=1 \text{ together imply } x=y.\tag{P'}
	\end{align*}

Moreover, if the paraorthomodular poset $\mathbf P$ is lattice-ordered then it is called a \emph{paraorthomodular lattice}.

An orthogonal paraorthomodular poset is called {\em sharply paraorthomodular} (see \cite{CFLLP}). Of course, any paraorthomodular lattice is sharply paraorthomodular.

It is worth noticing that every { orthomodular poset}  is paraorthomodular, but there are examples of paraorthomodular posets (or even lattices, indeed) which are not orthomodular, see \cite{CFLLP}.


\begin{center}
\begin{tabular}{c c c c}
\phantom{ccc}&
	\setlength{\unitlength}{6mm}
	\begin{picture}(4,8)
		\put(2,1){\circle*{.3}}
		\put(1,3){\circle*{.3}}
		\put(3,3){\circle*{.3}}
		\put(1,5){\circle*{.3}}
		\put(3,5){\circle*{.3}}
		\put(2,7){\circle*{.3}}
		\put(2,1){\line(-1,2)1}
		\put(2,1){\line(1,2)1}
		\put(1,5){\line(0,-1)2}
		\put(1,5){\line(1,2)1}
		\put(3,5){\line(0,-1)2}
		\put(3,5){\line(-1,2)1}
		\put(1.85,.3){$0$}
		\put(.35,2.85){$a$}
		\put(3.4,2.85){$b$}
		\put(.35,4.85){$a'$}
		\put(3.4,4.85){$b'$}
		\put(1.85,7.4){$1$}
		\put(-2.802,-.75){{\rm Fig.~1 \  A non-orthocomplemented }}
  \put(-2.802,-1.35){{\phantom{\rm Fig.~1} \  paraorthomodular lattice}}
	\end{picture}&\phantom{cccccccccccccc} &%
 	\setlength{\unitlength}{6mm}
	\begin{picture}(4,8)
		\put(2,1){\circle*{.3}}
		\put(1,3){\circle*{.3}}
		\put(3,3){\circle*{.3}}
		\put(1,5){\circle*{.3}}
		\put(3,5){\circle*{.3}}
		\put(2,7){\circle*{.3}}
		\put(2,1){\line(-1,2)1}
		\put(2,1){\line(1,2)1}
		\put(1,5){\line(0,-1)2}
		\put(1,5){\line(1,2)1}
		\put(3,5){\line(0,-1)2}
		\put(3,5){\line(-1,2)1}
		\put(1.85,.3){$0$}
		\put(.35,2.85){$a$}
		\put(3.4,2.85){$b$}
		\put(.35,4.85){$b'$}
		\put(3.4,4.85){$a'$}
		\put(1.85,7.4){$1$}
		\put(-2.802,-.75){{\rm Fig.~2 \  An orthocomplemented }}
  \put(-2.802,-1.35){{\phantom{\rm Fig.~2} \  non-orthomodular lattice}}
	\end{picture}
 \end{tabular}
\end{center}

\vspace*{4mm}

A bounded poset $(P,\leq,{}',0,1)$ with an antitone involution  is called 
an {\em orthocomplemented poset} if $x\vee x'=1$ for every $x\in P$.

An element $y \in P$ of a poset $\mathbf P=(P,\leq)$ is said to 
be a {\em complement} of $x \in P$ 
if $L(x, y) =L(P)$ and $U(x, y) = U(P)$. $\mathbf P$ is said to be 
{\em complemented} if each element of $P$ has a complement in $\mathbf P$. 

A {poset} $\mathbf P$ is called {\em distributive} if one of the following equivalent LU-identities is satisfied:
\begin{align*}
	L\big(U(x,y),z\big) & \approx LU\big(L(x,z),L(y,z)\big), \\
	U\big(L(x,z),L(y,z)\big) & \approx UL\big(U(x,y),z\big), \\
	U\big(L(x,y),z\big) & \approx UL\big(U(x,z),U(y,z)\big), \\
	L\big(U(x,z),U(y,z)\big) & \approx LU\big(L(x,y),z\big).
\end{align*}

$\mathbf P$ is said to be {\em Boolean} if it is distributive and complemented.

From \cite[Example 2.4]{CKL22}, we know that the poset 
depicted in Fig.~3 is a non-lattice Boolean poset which is not  an orthogonal poset since $a\perp c$, but $a\vee c$ does not exist. 

\vspace*{-1mm}

\begin{center}
	\setlength{\unitlength}{7mm}
	\begin{picture}(8,10)
		\put(4,1){\circle*{.3}}
		\put(1,3){\circle*{.3}}
		\put(3,3){\circle*{.3}}
		\put(5,3){\circle*{.3}}
		\put(7,3){\circle*{.3}}
		\put(1,5){\circle*{.3}}
		\put(7,5){\circle*{.3}}
		\put(1,7){\circle*{.3}}
		\put(3,7){\circle*{.3}}
		\put(5,7){\circle*{.3}}
		\put(7,7){\circle*{.3}}
		\put(4,9){\circle*{.3}}
		\put(4,1){\line(-3,2)3}
		\put(4,1){\line(-1,2)1}
		\put(4,1){\line(1,2)1}
		\put(4,1){\line(3,2)3}
		\put(4,9){\line(-3,-2)3}
		\put(4,9){\line(-1,-2)1}
		\put(4,9){\line(1,-2)1}
		\put(4,9){\line(3,-2)3}
		\put(1,3){\line(0,1)4}
		\put(1,3){\line(1,1)4}
		\put(3,3){\line(-1,1)2}
		\put(3,3){\line(1,1)4}
		\put(5,3){\line(-1,1)4}
		\put(5,3){\line(1,1)2}
		\put(7,3){\line(-1,1)4}
		\put(7,3){\line(0,1)4}
		\put(1,5){\line(1,1)2}
		\put(7,5){\line(-1,1)2}
		\put(3.85,.3){$0$}
		\put(.35,2.85){$a$}
		\put(2.35,2.85){$b$}
		\put(5.4,2.85){$c$}
		\put(7.4,2.85){$d$}
		\put(.35,4.85){$e$}
		\put(7.4,4.85){$e'$}
		\put(.35,6.85){$d'$}
		\put(2.35,6.85){$c'$}
		\put(5.4,6.85){$b'$}
		\put(7.4,6.85){$a'$}
		\put(3.85,9.4){$1$}
		\put(-0.32,-.75){{\rm Fig.~3 \quad  A non-orthogonal Boolean poset}}
	\end{picture}
\end{center}

\vspace*{4mm}

We have the following easy observation. 

\begin{lemma}\label{boiom}
	Every orthogonal Boolean  poset $\mathbf P$ is orthomodular. 
\end{lemma}
\begin{proof} Recall that the complementation on $\mathbf P$ is uniquely determined. Let us denote it 
	by ${}{'}$.
	Let $x, y\in P, x\leq y$. Then $x'\wedge y$ and $x\vee (x'\wedge y)$ exist. We compute:
	\begin{align*}
	L(y)&=L(1,y)=L(x\vee x',y)=L(U(x,x'),y)\\
	&=LU(L(x,y),L(x',y))=LU(x,x'\wedge y).
	\end{align*}
Since $U(y)=U(x,x'\wedge y)$,	we obtain that $y=x\vee (x'\wedge y)$.
\end{proof}

An involutive poset $\mathbf P$ is called {\em weakly Boolean} \cite{tkadlec} if for 
every $a, b \in P$, the condition 
$a\wedge b=a\wedge b'=0$ implies $ a= 0$.

Let $\mathbf P=(P,\leq)$ be a poset. In the following, for every 
subset $A$ of $P$, let $\Max A$ and $\Min A$ denote the set of all maximal and minimal elements of $A$, respectively. 

We say that a poset $\mathbf P$ is 
\begin{enumerate}
	\item {\em mub-complete} \cite{abramsky} if for every upper bound $x$
	of a finite  subset $M$ of $P$, there is a minimal upper bound of $M$ below $x$, 
	\item {\em lub-complete}  if for every lower bound $x$
	of a finite subset $M$ of $P$, there is a maximal lower bound of $M$ above $x$, 
	\item {\em mlub-complete}  if it is both mub-complete and lub-complete. 
\end{enumerate}

If $\mathbf P=(P,\leq,{}',0,1)$ is a bounded poset with an antitone involution, then 
$\mathbf P$ is mub-complete if and only if it is lub-complete.

Clearly, every finite poset, every poset of a finite height, and every lattice is mlub-complete.

The bounded poset $\mathbf P$ has the {\em property of maximality} \cite{tkadlec} if,
for every $a, b \in P$, the set $L(a,b)$ has a maximal element. Every lub-complete poset has the property of maximality.

From \cite[Theorem 3.1]{CKL22}, we know that 
the poset depicted in Fig.~4 is 
the smallest non-lattice orthomodular poset which is unique up to isomorphism. 

\begin{center}
	\setlength{\unitlength}{5.695mm}
	\begin{picture}(16,8)
		\put(8,1){\circle*{.3}}
		\put(1,3){\circle*{.3}}
		\put(3,3){\circle*{.3}}
		\put(5,3){\circle*{.3}}
		\put(7,3){\circle*{.3}}
		\put(9,3){\circle*{.3}}
		\put(11,3){\circle*{.3}}
		\put(13,3){\circle*{.3}}
		\put(15,3){\circle*{.3}}
		\put(1,5){\circle*{.3}}
		\put(3,5){\circle*{.3}}
		\put(5,5){\circle*{.3}}
		\put(7,5){\circle*{.3}}
		\put(9,5){\circle*{.3}}
		\put(11,5){\circle*{.3}}
		\put(13,5){\circle*{.3}}
		\put(15,5){\circle*{.3}}
		\put(8,7){\circle*{.3}}
		\put(1,3){\line(0,2)2}
		\put(1,3){\line(1,1)2}
		\put(1,3){\line(2,1)4}
		\put(1,3){\line(4,1)8}
		\put(3,3){\line(-1,1)2}
		\put(3,3){\line(0,1)2}
		\put(3,3){\line(2,1)4}
		\put(3,3){\line(4,1)8}
		\put(5,3){\line(-2,1)4}
		\put(5,3){\line(4,1)8}
		\put(7,3){\line(-3,1)6}
		\put(7,3){\line(4,1)8}
		\put(9,3){\line(-3,1)6}
		\put(9,3){\line(2,1)4}
		\put(11,3){\line(-4,1)8}
		\put(11,3){\line(2,1)4}
		\put(13,3){\line(-4,1)8}
		\put(13,3){\line(-3,1)6}
		\put(13,3){\line(0,1)2}
		\put(13,3){\line(1,1)2}
		\put(15,3){\line(-3,1)6}
		\put(15,3){\line(-2,1)4}
		\put(15,3){\line(-1,1)2}
		\put(15,3){\line(0,1)2}
		\put(8,1){\line(-7,2)7}
		\put(8,1){\line(-5,2)5}
		\put(8,1){\line(-3,2)3}
		\put(8,1){\line(-1,2)1}
		\put(8,1){\line(1,2)1}
		\put(8,1){\line(3,2)3}
		\put(8,1){\line(5,2)5}
		\put(8,1){\line(7,2)7}
		\put(8,7){\line(-7,-2)7}
		\put(8,7){\line(-5,-2)5}
		\put(8,7){\line(-3,-2)3}
		\put(8,7){\line(-1,-2)1}
		\put(8,7){\line(1,-2)1}
		\put(8,7){\line(3,-2)3}
		\put(8,7){\line(5,-2)5}
		\put(8,7){\line(7,-2)7}
		\put(7.85,.25){$0$}
		\put(.85,2.3){$a$}
		\put(2.85,2.3){$b$}
		\put(4.85,2.3){$c$}
		\put(6.85,2.3){$d$}
		\put(8.85,2.3){$e$}
		\put(10.85,2.3){$f$}
		\put(12.85,2.3){$g$}
		\put(14.85,2.3){$h$}
		\put(.8,5.4){$h'$}
		\put(2.8,5.4){$g'$}
		\put(4.8,5.4){$f'$}
		\put(6.8,5.4){$e'$}
		\put(8.8,5.4){$d'$}
		\put(10.8,5.4){$c'$}
		\put(12.8,5.4){$b'$}
		\put(14.8,5.4){$a'$}
		\put(7.85,7.4){$1$}
		\put(1.252,-.75){{\rm Fig.~4}\quad The smallest non-lattice orthomodular poset}
	\end{picture}
\end{center}
\vspace*{4mm}

Let $\mathbf P$ be an orthomodular poset. For every $a, b \in P$ with
$a \leq b$, let us denote $b-a=(b\wedge a')= (b' \vee a)' \in P$
(according to the definition of an orthomodular poset).
According to the orthomodular law, 
$b = a \vee (b-a)$ for every $a, b \in P$ with
$a \leq b$ and, moreover, $a \perp (b-a)$.

A pair $a, b\in P$ of elements of  an orthomodular poset $\mathbf P$ is 
called {\em compatible} if there exist three mutually orthogonal elements 
$c, d, e\in P$ such that $a=c\vee d$ and $b=d\vee e$. We write 
$a\mathrel{C}b$. 

Note that every compatible pair $a, b$ is contained in a Boolean algebra included in $\mathbf P$ with operations 
inherited from $\mathbf P$ (see {}\cite[Theorem 1.3.23]{PP}). 

Recall also that every weakly Boolean orthomodular poset with the property
of maximality is a Boolean algebra \cite[Theorem 4.2]{tkadlec}.

We sometimes extend an operator $*\colon P^2\rightarrow2^P$ to a binary operation on $2^P$ by
\[
A*B:=\bigcup_{x\in A,y\in B}(x*y)
\]
for all $A,B\in2^P$.

\section{Classical implication in posets} \label{classical}

In this section, we study the multiple-valued  classical implication in posets.

Let $(P,\leq,{}',0,1)$ be a lub-complete bounded poset  with an antitone involution. 
We introduce the following everywhere defined operator $\rightarrow_C\colon P^2\rightarrow2^P$ as follows:

\begin{align*}
	x\rightarrow_C y:=\Min U(x',y).\tag{IC}
	\end{align*}

We say that $\rightarrow_C$ is a {\em classical implication}.

For the operator $\rightarrow_C$ defined above, we can prove the following properties.

\begin{proposition}\label{bprop1}
	Let $(P,\leq,{}',0,1)$ be a lub-complete bounded poset  with an antitone involution  
	and $x,y, z\in P$. Then the following hold:
	\begin{enumerate}[{\rm(i)}]
		\item $y\leq x\rightarrow_C y$,
		\item $x\leq y$ implies $y\rightarrow_C z\leq_2 x\rightarrow_C z$,
			\item $x\rightarrow_C 1\approx 1$, $1\rightarrow_C x\approx x$,
		\item $x\rightarrow_C 0\approx x'$, $0\rightarrow_C x\approx 1$,
		\item If $x'\vee y$ exists then $x\rightarrow_C y=x'\vee y$.	
	\end{enumerate}
\end{proposition}
\begin{proof}
	\begin{enumerate}[(i):]
		\item Let $u\in \Min U(x',y)$. Then $y\leq u$.
	\item Let $x\leq y$ and $u\in \Min U(x',z)$. Since $y'\leq x'$ 
	we obtain $u\geq y',z$. Hence there is $v\in \Min U(y',z)$ such that $v\leq u$.
	\item We have $x\rightarrow_C 1= x'\vee 1=1$ and 
	$1\rightarrow_C x=0\vee x=x$.
	\item We compute $x\rightarrow_C 0= x'\vee 0=x'$ and 
	$0\rightarrow_C x= 1\vee x=1$. 
	\item It is transparent. 
	\end{enumerate}		
\end{proof}

\begin{theorem}\label{bnocmeg5}
		Let $\mathbf P=(P,\leq,{}',0,1)$ be an orthogonal 
	lub-complete poset. Then the following conditions are equivalent:
	\begin{enumerate}[{\rm(i)}]
		\item $\mathbf P$ is an orthocomplemented  poset.
		\item For all $x, y\in P$
			$$
		x\leq y \text{ implies	} x\rightarrow_C y=1.
		$$
	\end{enumerate}
\end{theorem}
\begin{proof} (i) $\Rightarrow$ (ii): Let $x, y\in P$. Assume that 
	$x\leq y$. Then  $\{1\}=U(x',x)\supseteq U(x',y)$. 
 Hence $x\rightarrow_C y=1$. 
 
 \noindent{}(ii) $\Rightarrow$ (i): Let $x\in P$. Then $x\leq x$ and 
 $x\rightarrow_C x=1$. We obtain $\{1\}=U(x',x)$, i.e., 
 $1= x\vee x'$. 
   Hence also $x\wedge x'=0$. 
\end{proof}

\begin{proposition}\label{bmeg1}
	Let $\mathbf P=(P,\leq,{}',0,1)$ be an orthogonal  lub-complete poset.
	Then the following conditions satisfy {\rm(i)} $\Rightarrow$ {\rm (ii)} $\Rightarrow$ {\rm(iii)}:
\begin{enumerate}[{\rm(i)}]
		\item $\mathbf P$ is a Boolean  poset.
		\item For all $x, y\in P$
			$$
		 x\rightarrow_C y=1 \text{ implies	} x\leq y.
		$$
  \item  $\mathbf P$ is a weakly Boolean  paraorthomodular poset.
		\end{enumerate}
	
\end{proposition}
\begin{proof}  {\rm(i)} $\Rightarrow$ {\rm (ii)}:
	Let $x, y\in P$ and $x\rightarrow_C y=1$. Then $x'\vee y=1$.  Hence 
	\begin{align*}
		L(x)&=L(U(x',y), x)=LU(L(x',x), L(y,x))=L(x,y)\subseteq L(y).
	\end{align*}
	
	We conclude that $x\leq y$.

 \noindent{}{\rm(ii)} $\Rightarrow$ {\rm (iii)}: Let $x, y\in P$, $x\leq y$ and $x\vee y'=1$. Then 
 $y\rightarrow_C x=1$. From (ii) we conclude that $y\leq x$, i.e., $x=y$ and 
 $\mathbf P$ is paraorthomodular.

 Assume now that $x\wedge y=x\wedge y'=0$. Then $x'\vee y'=1=x'\vee y$. Hence 
  $x\rightarrow_C y=1=x\rightarrow_C y'$. We conclude that $x\leq y$ and $x\leq y'$, i.e., $1=x'\vee y'\leq x'$. 
  Therefore $x=0$. 
\end{proof}	

Given the above result, it is natural to ask whether some converse of it holds. The answer is in the following theorem.

\begin{theorem}\label{bmeg3}
	 Let $\mathbf P=(P,\leq,{}',0,1)$ be an orthogonal 
	lub-complete poset. Then the following conditions are equivalent:
	\begin{enumerate}[{\rm(i)}]
		\item $\mathbf P$ is a Boolean  poset.
		\item For all $x, y\in P$
        $$
	x\leq y \text{ if and only if	} x\rightarrow_C y=1.
	$$
	\item $\mathbf P$ is 
	a Boolean algebra.
	\end{enumerate}
\end{theorem}
\begin{proof}  {\rm(i)} $\Rightarrow$ {\rm (ii)}: Since $\mathbf P$ is a Boolean  poset we obtain 
from Theorem \ref{bnocmeg5} and Proposition \ref{bmeg1} that, for all $x, y\in P$, 
  $$
	x\leq y \text{ if and only if	} x\rightarrow_C y=1.
	$$

\noindent{}{\rm(ii)} $\Rightarrow$ {\rm (iii)}: Again from Theorem \ref{bnocmeg5} and Proposition \ref{bmeg1} we obtain that 
$\mathbf P$ is an orthocomplemented paraorthomodular poset, i.e., 
	$\mathbf P$ is an orthomodular  poset.
	
	Since $\mathbf P$ is also weakly Boolean poset that has the maximality 
	property we conclude 
	from \cite[Theorem 4.2]{tkadlec} that $\mathbf P$ is a Boolean algebra. 

 \noindent{}{\rm(ii)} $\Rightarrow$ {\rm (iii)}: It is transparent.
\end{proof}			

Having defined classical implication $\rightarrow_C$, the question arises if there exists a binary connective $\odot_C$ such that they form a so-called {\em adjoint pair}, i.e.
\begin{align}
x\odot_C y\leq_1 z\text{ if and only if }x\leq_2 y\rightarrow_C z. \tag{AP}
\end{align}
We can state and prove the following result showing that condition (AP) 
 already yields the structure of a Boolean algebra.

\begin{proposition}\label{badji}
	Let $\mathbf P=(P,\leq,{}',0,1)$ be an orthogonal  orthocomplemented  lub-complete poset. 
	Then the following conditions are equivalent:
	\begin{enumerate}[{\rm(i)}]
		\item $\mathbf P$ is a Boolean algebra.
		\item There exists a binary operator $\odot_C$ such that 
		$x\odot_C y\leq_1 z$ if and only if 
		$x\leq_2 y\rightarrow_C z$ for all $x, y, z\in P$.
	\end{enumerate}
\end{proposition}
\begin{proof} (i) implies (ii):  It is well known that  
	$\odot_I=\wedge$ and $x\rightarrow_C y=x'\vee y$ realize the required adjunction.

(ii) implies (i):  Let $x, y\in P$. Then $x \odot_C y\leq x, y$. Namely, 
$x\leq 1=y'\vee y$ implies $x \odot_C y\leq y$ and $u\in \Min U(y',x)$ implies 
$x\leq u$, i.e., $x\odot_C y\leq x$.

Now, let $x, y\in P, x\leq y$. Then $x\vee (x'\wedge y)\leq y$ and 
$y \odot_C x'\leq_1 y\wedge x'$. Hence $y\leq x\vee (x'\wedge y)\leq y$, i.e., 
$\mathbf P$ is an orthomodular  poset. 

Let us show that $\mathbf P$ is weakly Boolean. Suppose that 
$x, y\in P$, $x\wedge y=0$ and $x\wedge y'=0$. Then 
$x \odot_C y\leq_1 0$ and $x \odot_C y'\leq_1 0$. From (AP) and 
Proposition \ref{bprop1} (iv), we obtain that $x\leq y'$ and $x\leq y$, i.e., 
$x=0$. We conclude  by \cite[Theorem 4.2]{tkadlec} that $\mathbf P$ is a Boolean algebra.

\end{proof}
A direct consequence of Propositions \ref{bmeg1}, \ref{bmeg3}, and \ref{badji} is:

\begin{theorem}\label{bmeg4}
		Let $\mathbf P=(P,\leq,{}',0,1)$ be an orthogonal 
	lub-complete poset.  Then the following conditions are equivalent:
	\begin{enumerate}[{\rm(i)}]
		\item $\mathbf P$ is a Boolean algebra.
		\item $\mathbf P$ satisfies 
		$$
		x\leq y \text{ if and only if	} x\rightarrow_C y=1
		$$
		for all $x, y\in P$.
		\item $\mathbf P$ is an orthocomplemented  poset such that 
		 there exists a binary operator $\odot_C$ such that 
		$x\odot_C y\leq_1 z$ if and only if 
		$x\leq_2 y\rightarrow_C z$ for all $x, y, z\in P$.
	\end{enumerate}
\end{theorem}

\section{Kalmbach implication in posets} \label{Kalman}

This section aims to study two multiple-valued  quantum implications - Kalmbach implication and non-tolens implication - in posets.

Let $(P,\leq,{}',0,1)$ be an orthogonal lub-complete poset. 
If $y\in P$ and $A\subseteq P$, then we will denote the set 
$\{y\vee a\mid a\in A\}$ by $y\vee A$ provided  all the elements $y\vee a$ exist. Analogously, we proceed in similar cases. Now we introduce the operator 

$\rightarrow_K\colon P^2\rightarrow2^P$ as follows:
\begin{align*}x\rightarrow_K y:=\Max L(x',y)\vee\,\,\Max L(x',y')\vee\,\, (x\wedge \Min U(x',y)). \tag{IK}
\end{align*}

$\Max L(x',y)\vee \Max L(x',y')$ is correctly defined since $\Max L(x',y')\perp \Max L(x',y)$.  Similarly, $x'\leq  \Min U(x',y)$ implies that $(x\wedge \Min U(x',y)$ 
is  correctly defined. From $\Max L(x',y)\vee \Max L(x',y')\leq x'$ and 
$x\wedge \Min U(x',y)\leq x$, we conclude that $x\rightarrow_K y$ is correctly defined. 

Similarly, as in lattices, we introduce the operator 
$\rightarrow_N\colon P^2\rightarrow2^P$ as follows:

\begin{align*}
	 x\rightarrow_N y:=y'\rightarrow_K x'. \tag{IN}
\end{align*}

We say that $\rightarrow_K$ is a {\em Kalmbach implication} 
and $\rightarrow_N$ is a {\em non-tolens implication}.

We first list several elementary properties of the operator $\rightarrow_K$  in orthogonal  lub-complete posets.

\begin{proposition}\label{prop1}
	Let $(P,\leq,{}',0,1)$ be an orthogonal lub-complete poset  
	and $x,y, z\in P$. Then the following hold:
	\begin{enumerate}[{\rm(i)}]
		\item \begin{tabular}{@{}l}\\
			$x\rightarrow_K y\approx\left\{
		\begin{array}{@{}ll}
			x\vee y'  \vee \Max L(x',y)         & \text{ if }x\leq y, \\
			y\vee  \Max L(x',y') & \text{ if }x\perp y, \\
			x'\vee  (x\wedge \Min U(x',y))          & \text{ if }y\leq x,
		\end{array}
		\right.$\\
		\\
	\end{tabular}
		\item $x\rightarrow_K 1\approx x\vee x'$, $1\rightarrow_K x\approx x$,
		\item $x\rightarrow_K 0\approx x'$, $0\rightarrow_K x\approx x\vee x'$.
	\end{enumerate}
\end{proposition}
\begin{proof}
	\begin{enumerate}[(i):]
		\item Assume $x\leq y$.  Then $y'\leq x'$ and hence $\Max L(x',y')=\{y'\}$, 
		$x\wedge \Min U(x',y)=\{x\}$. We conclude 
		 $x\rightarrow_K y=x\vee y'\vee \Max L(x',y)$.

		Assume now $x\leq y'$. Hence $y\leq x'$ and 
		$x\rightarrow_K y=y\vee  \Max L(x',y') \vee x\wedge\{x'\}=y\vee  \Max L(x',y')$. 
		
		Let $y\leq x$. Then $x\rightarrow_K y=x'\vee  \{0\} \vee (x\wedge \Min U(x',y))$.
		\item From (i) we have $x\rightarrow 1=x\vee  0 \vee \{x'\}=x\vee x'$ and 
		 $1\rightarrow x= 0\vee(1\wedge \{x\})=\{x\}$. 
		 
		 \item Again from (i) we compute $x\rightarrow 0=x'\vee  (x\wedge \{x'\})=\{x'\}$ and 
		 $0\rightarrow x= 0\vee x'\vee (\{1\wedge x\})=x\vee x'$. 
			\end{enumerate}
\end{proof}

We immediately obtain the following. 

\begin{corollary}\label{cor1N}
		Let $(P,\leq,{}',0,1)$ be an orthogonal lub-complete poset  
	and $x,y, z\in P$. Then the following hold:
	\begin{enumerate}[{\rm(i)}]
		\item \begin{tabular}{@{}l}\\
			$x\rightarrow_N y\approx\left\{
		\begin{array}{@{}ll}
			y\vee  (y'\wedge \Min U(x',y))         & \text{ if }x\leq y, \\
			x'\vee  \Max L(x,y) & \text{ if }x'\perp y', \\
			x\vee y'  \vee \Max L(x',y)          & \text{ if }y\leq x,
		\end{array}
		\right.$\\
		\\
		\end{tabular}
		
		\item $x\rightarrow_N 1\approx x\vee x'$, $1\rightarrow_N x\approx x$,
		\item $x\rightarrow_N 0\approx x'$, $0\rightarrow_N x\approx x\vee x'$.
	\end{enumerate}
\end{corollary}

We have exact parallels to the properties of classical implication, 
but often the proofs
are trickier.

\begin{theorem}\label{knocmeg5}
		Let $\mathbf P=(P,\leq,{}',0,1)$ be an orthogonal 
	lub-complete poset. Then the following conditions are equivalent:
	\begin{enumerate}[{\rm(i)}]
		\item $\mathbf P$ is an orthocomplemented  poset.
		\item For all $x, y\in P$
			$$
		x\leq y \text{ implies	} x\rightarrow_K y=1.
		$$
	
		\item For all $x, y\in P$
		$$
		x\leq y \text{ implies	} x\rightarrow_N y=1.
		$$
	\end{enumerate}
\end{theorem}
\begin{proof} (i) $\Rightarrow$ (ii): Let $x, y\in P$. Assume that 
	$x\leq y$. Then 
 \begin{align*}
  x\rightarrow_K y= x\vee y'\vee \Max L(x',y) = %
  (x\vee y')\vee (x'\wedge y)=1.  
 \end{align*}
 (ii) $\Rightarrow$ (iii): Let $x, y\in P$. Suppose that 
	$x\leq y$. Then $y'\leq x'$ and 
 \begin{align*}
     x\rightarrow_N y=y'\rightarrow_K x'= 1. 
 \end{align*}
   (iii) $\Rightarrow$ (i):   Let $x\in P$. Then $x\leq x$ and 
   \begin{align*}
       1= x\rightarrow_N x=x\vee  (x'\wedge \Min U(x',x))   =x\vee x'.
   \end{align*}
   Hence also $x\wedge x'=0$. 
\end{proof}

\begin{theorem}\label{meg1} 
		Let $\mathbf P=(P,\leq,{}',0,1)$ be an orthogonal 
	lub-complete poset. Then the following conditions are equivalent:
	\begin{enumerate}[{\rm(i)}]
		\item $\mathbf P$ is a paraorthomodular  poset.
		\item For all $x, y\in P$
			$$
		x\rightarrow_K y=1 \text{ implies	} x\leq y.
		$$
	
		\item For all $x, y\in P$
		$$
		x\rightarrow_N y=1 \text{ implies	} x\leq y.
		$$
	\end{enumerate}
\end{theorem}
\begin{proof} (i) $\Rightarrow$ (ii): Let $x, y\in P$. Assume that 
	$x\rightarrow_K y=1$. Then for all $z_1, z_2, z_3\in P$ such that 
	$z_1\in \Max L(x',y), z_2\in\Max L(x',y'), z_3\in \Min U(x',y)$ we have 
	$z_1\vee z_2\vee (x\wedge z_3)=1$. Since $z_1\perp z_2$ we obtain that 
         $z_1\vee z_2$ exists and $z_1\vee z_2\leq x'$. We conclude that 
         $(z_1\vee z_2)\vee x=1$. From (P') we have $z_1\vee z_2= x'$.

         Similarly, from $x\wedge z_3\leq x$ and $x'\vee (x\wedge z_3)=1$ we get from 
         (P') that $x\wedge z_3= x$, i.e., $x\leq z_3$. Therefore $x\vee x'=1$ and $x\wedge x'=0$.

         Since $z_2\perp x$ we obtain that 
         $z_2\vee x$ exists and $z_1\leq x'\wedge z_2'$. From $z_1\vee (z_2\vee z)=1$ and (P') we conclude 
         $z_1=x'\wedge z_2'$. Similarly, $z_2=x'\wedge z_1'$. But $z_1\in \Max L(x',y)$ 
         and $z_2\in\Max L(x',y')$ were chosen arbitrarily. Hence $z_1=x'\wedge y$ and 
         $z_2=x'\wedge y'$. We then have $1=(x'\wedge y')\vee (x'\wedge y)\vee x$. 

         Since $(x'\wedge y)\vee x\leq x\vee y$ and 
         $0=(x\vee y)\vee (x\vee y')\wedge x'$ we obtain from (P) that 
         $(x'\wedge y)\vee x= x\vee y$.  From the latter equality, 
         $x'\wedge y'\leq y'$ and (P') we conclude that  $x'\wedge y'= y'$. 
         Hence  $ y'\leq x'$, i.e., $x\leq y$.
 
 \noindent{}(ii) $\Rightarrow$ (iii): Let $x, y\in P$. Suppose that 
	$x\rightarrow_N y=1$. Then $y'\rightarrow_K x'=1$ and we obtain 
 $y'\leq x'$. Hence  $x\leq y$.

\noindent{}(iii) $\Rightarrow$ (i):   Let  $x\leq y$ and 
$x\vee y'=1$. Then  $x\perp y'$ and from 
Corollary \ref{cor1N} we conclude $1=x\vee y'\leq y'\vee  (y\wedge (y'\vee x)) =y\rightarrow_K x=x'\rightarrow_N y'$. 
Hence  $x'\leq y'$, i.e., $x=y$.
\end{proof}

As a conclusion of the preceding, we immediately obtain the following.

\begin{theorem}\label{meg5}
		Let $\mathbf P=(P,\leq,{}',0,1)$ be an orthogonal 
	lub-complete poset. Then the following conditions are equivalent:
	\begin{enumerate}[{\rm(i)}]
		\item $\mathbf P$ is an orthomodular  poset.
		\item $\mathbf P$ satisfies 
			$$
		x\leq y \text{ if and only if	} x\rightarrow_K y=1
		$$
		for all $x, y\in P$.
		\item $\mathbf P$ satisfies 
		$$
		x\leq y \text{ if and only if	} x\rightarrow_N y=1
		$$
		for all $x, y\in P$.
	\end{enumerate}
\end{theorem}

What we do not have is a characterization of orthomodular posets 
satisfying (AP) for Kalmbach implication or non-tolens implication, respectively. 
All we can say is the following.

\begin{proposition}\label{Kadji}
	Let $\mathbf P=(P,\leq,{}',0,1)$ be an orthogonal lub-complete poset
	and assume  that there exists a binary operator $\odot_K$ such that 
	$x\odot_K y\leq_1 z$ if and only if 
	$x\leq_2 y\rightarrow_K z$ for all $x, y, z\in P$.
	Then 
	\begin{enumerate}[{\rm(i)}]
		\item $x\odot_K x'=\{0\}$ for all $x \in P$, 
		\item $\mathbf P$ is an orthomodular poset. 
	\end{enumerate}
\end{proposition}
\begin{proof} (i): Let $x\in P$. Then 
	$x\leq_2 \{x\}= x'\rightarrow_K 0$. Hence also 
	$x\odot_K x'\leq_1 \{0\}$, i.e., $x\odot_K x'=\{0\}$. 
	
	(ii): Let $x, y\in P$, $x\leq y$. From (i), we know that 
	$y\odot_K y'=\{0\}\leq_1 x$. Since $x\perp y'$, we 
	obtain from Proposition \ref{prop1}, (i) that 
	$y\leq_2 y'\rightarrow_K x=x\vee \Max L(y, x')=\{x\vee (y\wedge x')\}$. 
	Hence 	$y\leq x\vee (y\wedge x')\leq y$, i.e., 
	$y= x\vee (y\wedge x')$.

\end{proof}

\section{Sasaki and Dishkant implications}\label{SasDis}

In this section, we study the Sasaki and Dishkant  implications for {orthogonal lub-complete posets} introduced in 
\cite{CFLLP2} and investigate their properties following our approach from Section \ref{Kalman}.

\begin{definition} \label{defSasSis} {\rm \cite{CFLLP2}}
Let $(P,\leq,{}',0,1)$ be an {orthogonal lub-complete poset} and let us introduce two operators $\rightarrow_S$ 
and $\rightarrow_D$ as follows:
	\begin{align*}
		{x\rightarrow_S y}:= x'\vee\, \Max L(x,y) \,  \tag{IS}
	\end{align*}
	
	and also
	\begin{align*}
		{x\rightarrow_D y} := {y'\rightarrow_S x'}=%
  y\vee\, \Max L(x',y'). \tag{ID}
	\end{align*}
	We say that $\rightarrow_S$ is a \textit{Sasaki implication} 
	and $\rightarrow_D$ is a \textit{Dishkant implication}.
\end{definition}

Then we obtain the following properties of Sasaki and Dishkant implications.

\begin{proposition}\label{SasDisprop}
	Let $(P,\leq,{}',0,1)$ be an {orthogonal lub-complete poset}  
	and $x,y, z\in P$. Then:
	\begin{enumerate}[{\rm(i)}]
		\item $x'\leq x\rightarrow_S y$,
		\item $x\leq y$ implies $z\rightarrow_S x\leq_1 z\rightarrow y$,
		\item \begin{tabular}{@{}l}\\
			$x\rightarrow_S y\approx\left\{
			\begin{array}{@{}ll}
				x\vee x'          & \text{ if }\/x\leq y, \\
				x'\vee  (x\wedge y) & \text{ if }\/x'\perp y', \\
				x'\vee  y=x\rightarrow_D y & \text{ if }\/y\leq x,
			\end{array}
			\right.$\\
			\\
		\end{tabular}
		\item $x\rightarrow_S 1\approx x\vee x'$, $1\rightarrow_S x\approx x$,
		\item $x\rightarrow_S 0\approx x'$, $0\rightarrow_S x\approx 1$.
	\end{enumerate}
\end{proposition}
\begin{proof}
\begin{enumerate}[(i)]
\item follows directly from (IS).
\item Assume $x\leq y$ and $z\in P$.  Let $u\in y\rightarrow_S z$. 
Then there exists some $w\in\Max L(x,z)$ with $z'\vee w=u$. Since 
$x\leq y$ we have $L(x,z)\subseteq L(y,z)$. Thus there exists 
some $t\in\Max L(y,z)$ with $w\leq t$. Now
\[
u=z'\vee w\leq z'\vee t\in z\rightarrow_S y
\]
proving $z\rightarrow_S x\leq_1 z\rightarrow_S y$.
\item $x\rightarrow_S y\approx x'\vee\Max L(x,y)\approx\left\{
\begin{array}{ll}
x\vee x' & \text{ if }x\leq y, \\
x'\vee(x\wedge y) & \text{ if }x\perp y, \\
x'\vee y=x\rightarrow_D y & \text{ if }y\leq x.
\end{array}
\right.$
\end{enumerate}

\noindent{}(iv) and (v) follow directly from (iii).
\end{proof}

From the definition of $\rightarrow_D$ and 
Proposition \ref{SasDisprop} we immediately obtain 
the following. 

\begin{corollary}{\rm \cite[Theorem 1, (i)-(iii)]{CFLLP2}}\label{corth1}
Let $(P,\leq,{}',0,1)$ be an orthogonal lub-complete poset and $x,y\in P$. Then the following hold:
\begin{enumerate}[{\rm(i)}]
\item $y\leq x\rightarrow_D y$,
\item $x\leq y$ implies $y\rightarrow_D z\leq_1 x\rightarrow_D z$,
\item $x\rightarrow_D y\approx\left\{
\begin{array}{ll}
y\vee y'           & \text{ if }\/x\leq y, \\
y\vee(x'\wedge y') & \text{ if }\/x\perp y, \\
x'\vee y           & \text{ if }\/y\leq x,
\end{array}
\right.$
\item $x\rightarrow_D 1\approx 1$, $1\rightarrow_D x\approx x$,
\item $x\rightarrow_D 0\approx x'$, $0\rightarrow_D x\approx x\vee x'$.
\end{enumerate}
\end{corollary}

\begin{theorem}\label{ocmeg5}
		Let $\mathbf P=(P,\leq,{}',0,1)$ be an orthogonal 
	lub-complete poset. Then the following conditions are equivalent:
	\begin{enumerate}[{\rm(i)}]
		\item $\mathbf P$ is an orthocomplemented  poset.
		\item For all $x, y\in P$
			$$
		x\leq y \text{ implies	} x\rightarrow_S y=1.
		$$
	
		\item For all $x, y\in P$
		$$
		x\leq y \text{ implies	} x\rightarrow_D y=1.
		$$
	\end{enumerate}
\end{theorem}
\begin{proof} (i) $\Rightarrow$ (ii): Let $x, y\in P$. Assume that 
	$x\leq y$. Then 
 \begin{align*}
  x\rightarrow_S y= x\vee x'=1.  
 \end{align*}
 (ii) $\Rightarrow$ (iii): Let $x, y\in P$. Suppose that 
	$x\leq y$. Then $y'\leq x'$ and 
 \begin{align*}
     x\rightarrow_D y=y'\rightarrow_S x'= 1. 
 \end{align*}
   (iii) $\Rightarrow$ (i):   Let $x\in P$. Then $x\leq x$ and 
   \begin{align*}
       1= x\rightarrow_D x=x\vee x'.
   \end{align*}
   Hence also $x\wedge x'=0$. 
\end{proof}

Motivated by Theorem \cite[Theorem 2]{CFLLP2} we obtain the following 
extension of it. 

\begin{theorem}\label{paraocmeg5} {\rm \cite[Theorem 2]{CFLLP2}}
		Let $\mathbf P=(P,\leq,{}',0,1)$ be an orthogonal 
	lub-complete poset. Then the following conditions are equivalent:
	\begin{enumerate}[{\rm(i)}]
		\item $\mathbf P$ is a paraorthomodular  poset.
		\item For all $x, y\in P$
			$$
		x\rightarrow_S y=1 \text{ implies	} x\leq y.
		$$
	
		\item For all $x, y\in P$
		$$
		x\rightarrow_D y=1 \text{ implies	} x\leq y.
		$$
	\end{enumerate}
\end{theorem}
\begin{proof} (i) $\Rightarrow$ (ii): Let $x, y\in P$. Assume that 
	$x\rightarrow_S y=1$. Then there is $z\in L(x,y)$ such that 
 $x'\vee z=1$. Hence $x\wedge z'=0$ and $z\leq x$. Since 
 $\mathbf P$ is a paraorthomodular  poset we conclude that $x=z\leq y$.

 \noindent{}(ii) $\Rightarrow$ (iii): Let $x, y\in P$. Suppose that 
	$x\rightarrow_D y=1$. Then $y'\rightarrow_S x'=1$ and we obtain 
 $y'\leq x'$. Hence  $x\leq y$.

\noindent{}(iii) $\Rightarrow$ (i):   Let  $x\leq y$ and 
$x'\wedge y=0$. Then  $y'\leq x'$ and from 
Corollary \ref{corth1} we conclude $1=x\vee y'=x'\rightarrow_D y'$. 
Hence  $x'\leq y'$, i.e., $x=y$.
\end{proof}

As a conclusion of Theorems \ref{ocmeg5} and \ref{paraocmeg5}, we obtain the following.

\begin{theorem}\label{parocdsmeg5}
		Let $\mathbf P=(P,\leq,{}',0,1)$ be an orthogonal 
	lub-complete poset. Then the following conditions are equivalent:
	\begin{enumerate}[{\rm(i)}]
		\item $\mathbf P$ is an orthomodular  poset.
		\item $\mathbf P$ satisfies 
			$$
		x\leq y \text{ if and only if	} x\rightarrow_S y=1
		$$
		for all $x, y\in P$.
		\item $\mathbf P$ satisfies 
		$$
		x\leq y \text{ if and only if	} x\rightarrow_D y=1
		$$
		for all $x, y\in P$.
	\end{enumerate}
\end{theorem}

\section{Modus ponens for quantum implications}\label{modus ponens}

We now present the appropriate modus ponens
rule in the following way.

\begin{theorem}\label{modus}
	Let $\mathbf P=(P,\leq,{}',0,1)$ be an orthogonal 
	lub-complete poset and $\rightarrow\in \{\rightarrow_C, \rightarrow_K, \rightarrow_N, \rightarrow_S, \rightarrow_D\}$. Then the following conditions are equivalent:
	\begin{enumerate}[{\rm(i)}]
		\item \begin{align*}
			x\leq y\ \&\ x\rightarrow y\leq u\rightarrow v\ 
			\Rightarrow\ u\leq v\tag{MPO}
		\end{align*}
			for all $x, y, u, v\in P$.
		\item 
		\begin{align*}
		x\leq y \text{ if and only if	} x\rightarrow y=1\tag{OP}
			\end{align*}
		for all $x, y\in P$.
	\end{enumerate}
\end{theorem}
\begin{proof} (i) $\Rightarrow$ (ii): Let $x, y\in P$. Assume first that 
	$x\leq y$. From propositions and corollaries \ref{bprop1}, \ref{prop1},  \ref{cor1N}, \ref{SasDisprop}, and \ref{corth1}, 
	respectively, we obtain that $x\rightarrow y=1\rightarrow (x\rightarrow y)$.
	Hence $1\leq x\rightarrow y$ by (MPO). Suppose now that $1\leq x\rightarrow y$. 
	Then $1=1\rightarrow 1$ and $1\leq 1$. Again by (MPO), we have $x\leq y$. 
	
	 (ii) $\Rightarrow$ (i): Let $x, y, u, v\in P$. Assume that 
	 $x\leq y$ and $x\rightarrow y\leq u\rightarrow v$. Then by (OP), we have 
	 $x\rightarrow y=1$. Hence $u\rightarrow v=1$, i.e., again by (OP), we conclude 
	 that $u\leq v$.	
\end{proof}

An immediate consequence of Theorems \ref{bmeg4}, \ref{meg5}, 
\ref{parocdsmeg5}, and 
\ref{modus} is:

\begin{theorem}\label{meg6}
	Let $\mathbf P=(P,\leq,{}',0,1)$ be an orthogonal 
	lub-complete poset and  $\rightarrow\in \{\rightarrow_K, \rightarrow_N, \rightarrow_S, \rightarrow_D\}$. Then the following holds. 
	\begin{enumerate}[{\rm(i)}]
		\item $\mathbf P$ is Boolean if and only if  {\rm(MPO)} holds for $\rightarrow_C$ in $\mathbf P$.
		\item $\mathbf P$ is an orthomodular  poset if and only if  {\rm (MPO)} holds for $\rightarrow$ in $\mathbf P$.
		\end{enumerate}
\end{theorem}

\section{Conclusions}\label{Conclusions}

Our study shows that in the same way as for binary orthologics in \cite{mphpa98}, 
the assumption of order property or modus ponens rule for classical implication turns orthologic into classical logic. Similarly, the premise of order property 
or modus ponens rule for Kalmbach implication 
or non-tollens implication   or Dishkant implication  or Sasaki  implication turns orthologic into quantum logic.

For future work, there are a number of open problems that we plan to address.
In particular, we can mention the following ones: 

\begin{enumerate}
	\item We intend to represent those orthomodular posets such that (AP) holds 
	for $\rightarrow_K$ or $\rightarrow_N$, respectively.
	\item  We intend to study whether the order property of the 
	remaining one quantum implication introduced by Kalmbach 
	guarantees that we obtain a quantum logic.
\end{enumerate}

\section*{Acknowledgment}

Support of the research of all authors by the Austrian Science Fund (FWF), project I~4579-N, and the Czech Science Foundation (GA\v CR), project 20-09869L, entitled ``The many facets of orthomodularity'' is gratefully acknowledged. The first author was also supported by the project MUNI/A/1099/2022 by Masaryk University.


\end{document}